\documentclass[12pt]{article}

\usepackage{mathrsfs,amsmath,amsthm,amsfonts,amssymb,amscd,cite}
\usepackage{graphicx}
\usepackage{epstopdf}
\usepackage{color}
\usepackage{hyperref}
\usepackage{tikz}
\usepackage{comment}
\usetikzlibrary{positioning, calc, decorations.pathreplacing}

\usepackage{subfigure} 

\allowdisplaybreaks[4]

\theoremstyle{definition}

\newtheorem {theorem}{Theorem}[section]
\newtheorem {lemma}[theorem]{Lemma}
\newtheorem {corollary}[theorem]{Corollary}
\newtheorem {conjecture}[theorem]{Conjecture}

\newtheorem{remark}[theorem]{Remark}

\newcommand{\E}{\mathbb{E}}

\newcommand{\Var}{\operatorname{Var}}
\newcommand{\Cov}{\operatorname{Cov}}

\usepackage{fullpage}

\title{On the Variance Fraction of the Hard-Core Model on Graphs with Bounded Maximum Degree}

\author {Weiyuan Zhang,
~~Kexiang Xu\thanks{Corresponding author. Email addresses: \texttt{weiyuanzhang@nuaa.edu.cn} (W. Zhang),
\texttt{kexxu1221@126.com} (K. Xu).   }\\\\
\small  School of Mathematics, Nanjing University of Aeronautics and Astronautics,\\
\small Nanjing, Jiangsu, 210016, PR China}

\date{}

\begin{document}
\maketitle

\begin{abstract}

The hard-core model can be used to understand the number of independent sets in graphs in extremal graph theory.
The occupancy fraction, defined by Davies \textit{et al.} in 2017 as the logarithmic derivative of the independence polynomial of a graph, is a key quantity in the hard-core model.
The variance fraction, introduced by Davies \textit{et al.} in 2025, is defined as the derivative of the occupancy fraction with respect to the logarithm of the fugacity.
Since the occupancy fraction can be obtained by integrating the variance fraction with respect to the logarithm of the fugacity, bounding the variance fraction yields the corresponding bounds on the occupancy fraction.
Moreover,  the occupancy fraction correlates, in quantity, to the independence polynomial.

In this note we provide two lower bounds on the variance fraction, proving the conjecture by Davies \textit{et al.} in 2025, for graphs with bounded maximum degree and for graphs with $n$ vertices, respectively.
We also derive lower bounds for other graph classes, including graphs with a given edge chromatic number, $d$-regular graphs, and triangle-free graphs with bounded maximum degree.

\par\vspace{2mm}

\par\vspace{1mm}

\noindent{\bfseries  AMS Classification (2020): 05C31, 05C35, 05C69}
\end{abstract}

\section{Introduction}

\ \indent Throughout this paper we only consider finite, undirected and simple graphs.  Let $G$ be a graph with vertex set $V(G)$ and edge set $E(G)$.
The order of $G$ is the number of its vertices.
For a vertex $v\in V(G)$, we define the open neighborhood of $v$ as $N_G(v)= \{u\in V(G):uv\in E(G)\}$ and the closed neighborhood of $v$ as $N_G[v] = N_G(v)\cup\{v\}$.
The degree $d_G(u)$ of vertex $u$ is the cardinality of $N_G(u)$.
A set $S\subseteq V(G)$ in a graph $G$ is an \textit{independent set} if $G[S]$, the induced subgraph of $G$ by $S$, consists of $|S|$ isolated vertices. The \textit{independence number} $\alpha(G)$ of $G$ is the cardinality of a maximum independent set of $G$.
A \textit{clique} of a graph $G$ is an induced subgraph of $G$ that is complete.
The \textit{clique number} of $G$ is the cardinality of a maximum clique of $G$.
We denote by $K_n$ the complete graph on $n$ vertices, and by $K_{a,b}$ the complete bipartite graph whose two partite sets have sizes $a$ and $b$, respectively.
A proper vertex coloring of a graph $G$ is an assignment of colors to the vertices of $G$ such that any two adjacent vertices receive distinct colors. The \textit{chromatic number} of $G$, denoted by $\chi(G)$, is the minimum number of colors required for a proper vertex coloring of $G$.
A proper edge coloring of a graph $G$ is an assignment of colors to the edges of $G$ such that any two edges sharing a common endpoint receive distinct colors. The \textit{edge chromatic number} of $G$, denoted by $\chi'(G)$, is the minimum number of colors for a proper edge coloring of $G$. By convention, $\chi'(G)=0$ if $G$ has no edges.

The hard-core model is a probability distribution over independent sets in a graph, where a set $I$ is chosen with probability proportional to $\lambda^{|I|}$ for a positive real parameter $\lambda$. This is a widely applicable distribution that generalizes the uniform distribution (corresponding to $\lambda=1$) and possesses important properties related to conditional independence and entropy.
In extremal graph theory, the hard-core model can be used to understand the minimum and maximum numbers of independent sets in graphs.
Extremal bounds in this context have a long history and a few surprising applications, e.g. \cite{A1996,S1995,BM2019}.

Kahn \cite{K2001} states that the disjoint union of $\frac{n}{2d}$ copies of complete bipartite graph $K_{d,d}$ has the maximum number of independent sets among all $d$-regular bipartite graphs on $n$ vertices with $n$ divisible by $2d$.
Zhao \cite{Z2010} extended this result to all $d$-regular graphs.
In 2019, Sah \textit{et al.} \cite{SSSZ2019} determined the maximum number of independent sets in an irregular graph and extended their result to an upper bound in weighted form:
\[
P_G(\lambda)=\sum_{I\in \mathcal{I}(G)} \lambda^{|I|},
\]
where $\mathcal{I}(G)$ is the set of all independent sets of $G$, including the empty set.
Here $P_G(\lambda)$ is the \textit{partition function} of the hard-core model on $G$ (also called the independence polynomial of $G$).
Note that $P_G(\lambda)$ with $\lambda=1$  gives the total number $\sigma(G)$ of independent sets, including the empty set, of graph $G$.
An alternative version of $\sigma(G)$ is  $\sigma(G)=\sum\limits_{k=0}^{\alpha(G)}i_k(G)$ where $i_k(G)$ denotes the number of independent sets of size $k$ in $G$.
See \cite{Wei26+,Xu26,Xu22} for some novel results on $\sigma(G)$.
Let $G$ be a graph with $n$ vertices.
The hard-core model on $G$ is the measure $\mu_{G,\lambda}:\mathcal{I}(G)\rightarrow [0,1]$ given by
\[
\mu_{G,\lambda} (I):=\Pr[I]=\frac{\lambda^{|I|}}{P_G(\lambda)}.
\]

Borrowing some concepts from physics, Davies \textit{et al.} \cite{DJPR2017,DST2025} invented a set of terminologies. They defined the \textit{free energy density} as
\[
F_G(\lambda)=\frac{1}{n} \log P_G(\lambda),
\]
the \textit{occupancy fraction} as
\[
E_G(\lambda)=\lambda \frac{\partial}{\partial\lambda}F_G(\lambda),
\]
and the \textit{variance fraction} as
\[
V_G(\lambda)=\lambda \frac{\partial}{\partial\lambda}E_G(\lambda).
\]

Clearly, the variance fraction is the derivative of the occupancy fraction with respect to the natural logarithm of the fugacity:
\[
V_G(\lambda)=\lambda \frac{\partial E_G(\lambda)}{\partial \lambda} = \frac{\partial E_G(e^s)}{\partial s} = \frac{\partial E_G(\lambda)}{\partial \log\lambda},
\]
where $s=\log \lambda$.

We can see that
\[
n E_G(\lambda) = \E _{I\sim \mu_{G,\lambda}} (|I|) = \frac{\lambda P_{G}^{\prime}(\lambda)}{P_{G}(\lambda)}
\]
is the expected number of vertices of $G$ occupied by a random independent set from the hard-core model.
Similarly, we know that
\[
n V_G(\lambda) = \Var _{I\sim \mu_{G,\lambda}}(|I|) = \E _{I\sim \mu_{G,\lambda}} (|I|^2) - \left(\E _{I\sim \mu_{G,\lambda}} (|I|)\right)^2
\]
is the variance of the size of an independent set from the model.
If there is no confusion from context, we usually use $\E $ to represent $\E _{I\sim \mu_{G,\lambda}}$, and $\Var $ to represent $\Var _{I\sim \mu_{G,\lambda}}$.

Several bounds on the free energy density can be obtained from earlier examples of extremal results for the hard-core model.
If $G$ is a $d$-regular graph, then $F_G(\lambda)\leq F_{K_{d,d}}(\lambda)$ from combining several celebrated results of Kahn \cite{K2001} and Zhao \cite{Z2010}.
The lower bound $F_G(\lambda)\geq F_{K_{d+1}}(\lambda)$ follows from a result of Cutler and Radcliffe \cite{CR2014}.
If $d_u$ is the degree of vertex $u$ in graph $G$, then
$F_G(\lambda)\geq \frac{1}{n} \sum_{u\in V(G)}F_{K_{d_u+1}}(\lambda)$ can be  derived from a result by Sah \textit{et al.} \cite{SSSZ2019}.
 Other scholars have also done some excellent work on counting independent sets \cite{DP2023,JPP2023,Z2017,S2015,SSS2022,BBN2021}.

It is natural to consider the generalizations of the above results to the bounds on the parameters $E_G(\lambda)$ and $V_G(\lambda)$.
Davies \textit{et al.} \cite{DJPR2017} introduced a novel technique, called the occupancy method, and used it to prove the tight upper bounds on the independence polynomial and the matching polynomial of $d$-regular graphs.
In the subsequent work \cite{DJPR2018}, they considered the lower bounds on $E_G(\lambda)$ for triangle-free graphs with a given maximum degree, and derived the corresponding lower bounds on the number of independent sets using the properties of the occupancy fraction.
Cutler and Radcliffe \cite{CR2018} used the same method to obtain lower bounds on $E_G(\lambda)$ for $d$-regular graphs and, as a consequence, a tight lower bound on the independence polynomial for this class of graphs. They also obtained a lower bound on $E_G(\lambda)$ for $d$-regular triangle-free graphs.
Perarnau and Perkins \cite{PP2018} obtained an upper bound on $E_G(\lambda)$ for cubic graphs of girth at least $5$ and a lower bound on $E_G(\lambda)$ for triangle-free cubic graphs. They used these results to prove a tight upper bound and a lower bound, respectively, on the independence polynomial for these graph classes.
Cohen \textit{et al.} \cite{CPST2022} used the occupancy method to count the number of independent sets in uniform, regular, linear hypergraphs.
Davies \textit{et al.} \cite{DST2025} proved a lower bound on $E_G(\lambda)$ for (triangle-free) graphs with a given degree sequence, provided some bounds on $V_G(\lambda)$ for $n$-vertex graphs and posed the following conjecture:
\begin{conjecture}\cite{DST2025}
Let $G$ be an $n$-vertex graph.
Then $V_G(\lambda)\geq V_{K_n}(\lambda)$ for any $\lambda>0$.
If $G$ has maximum degree $\Delta$, then the lower bound can be improved to $V_G(\lambda)\geq V_{K_{\Delta+1}}(\lambda)$.
\end{conjecture}

The hard-core model can also be used to provide some bounds on certain off-diagonal Ramsey numbers and related quantities \cite{S1983}; some of these applications yield, asymptotically, the best results currently known to us. Its conditional independence properties can be exploited in probabilistic and algorithmic methods for graph coloring and related problems \cite{M2019,BKNP2022,DVKP2021,JMP2026,MS2025,GT2006,BLM2022,DVKP2020,SS2005}.

The following theorem was proved by Davies and Kang in \cite{DK2025}:

\begin{theorem}\cite{DK2025}\label{th:Ekdelta}
Let $G$ be a graph with maximum degree $\Delta$. Then, for any $\lambda>0$,
\[
E_G(\lambda)\geq E_{K_{\Delta+1}}(\lambda) = \frac{\lambda}{1+(\Delta+1)\lambda}.
\]
\end{theorem}

The following two results of ours resolve the conjecture of Davies \textit{et al.} \cite{DST2025}.

\begin{theorem}\label{th:kdelta}
Let $G$ be a graph with maximum degree $\Delta$. Then, for any $\lambda>0$, we have
\[
V_G(\lambda)\geq V_{K_{\Delta+1}}(\lambda) = \frac{\lambda}{(1+(\Delta+1)\lambda)^2},
\]
with equality if and only if $G$ is a disjoint union of copies of $K_{\Delta+1}$.
\end{theorem}

\begin{corollary}\label{th:kn}
Let $G$ be an $n$-vertex graph. Then, for any $\lambda>0$, we have
\[
V_G(\lambda)\geq V_{K_n}(\lambda) = \frac{\lambda}{(1+n\lambda)^2},
\]
with equality if and only if $G\cong K_{n}$.
\end{corollary}

 Theorem \ref{th:kdelta} can be viewed as an extension of Theorem \ref{th:Ekdelta}.
By Theorem \ref{th:kdelta}, we can get the following results directly.

\begin{corollary}\label{th:dregular}
Let $G$ be a $d$-regular graph on $n$ vertices. Then, for every $\lambda>0$,
\[
V_G(\lambda)\geq V_{K_{d+1}}(\lambda)
=\frac{\lambda}{\bigl(1+(d+1)\lambda\bigr)^2},
\]
with equality if and only if $G$ is a disjoint union of copies of $K_{d+1}$.
\end{corollary}

Corollary \ref{th:dregular} extends Theorem 3 in \cite{CR2018}, which states that the lower bound on $E_G(\lambda)$ is $E_{K_{d+1}}(\lambda)$ for a $d$-regular graph $G$ on $n$ vertices,
and the equality holds if and only if $G$ is a disjoint union of copies of $K_{d+1}$.

\begin{corollary}\label{chi'}
Let $G$ be a graph with edge chromatic number $\chi'(G)$. Then, for every $\lambda>0$,
$$
V_G(\lambda) \geq  \frac{\lambda}{\bigl(1+(\chi'(G)+1)\lambda\bigr)^2},
$$
with equality if and only if $G$ is a disjoint union of copies of $K_{\chi'(G)+1}$ where $\chi'(G)$ is either zero or odd.
\end{corollary}

The Lambert $W$-function is the multivalued inverse of the function
$w\mapsto we^w$. Thus, for every complex number $z$, every value $W(z)$ satisfies
\[
W(z)e^{W(z)}=z.
\]
Throughout this note, $W$ denotes the principal branch $W_0$. Since all
arguments considered here are nonnegative real numbers, for every
$z\geq0$, $W(z)$ is the unique nonnegative real solution $w$ of the
equation $W(z)e^{W(z)}=z$.

Davies \textit{et al.} \cite{DJPR2018} provided a lower bound on $E_G(\lambda)$ for triangle-free graphs with a given maximum degree.
Here we extend this result by proving a lower bound on $V_G(\lambda)$.

\begin{theorem}\label{thm:c3freedegree}
Let $G$ be a triangle-free graph with maximum degree $\Delta$. Then for any $\lambda > 0$,
\[
V_G(\lambda) \ge
\frac{\lambda}{(\Delta+1)(1+\lambda)^2}
\frac{W\!\left(\Delta\log(1+\lambda)\right)}{\Delta\log(1+\lambda)},
\]
where $W$ is the (positive real branch of the) Lambert $W$-function.
\end{theorem}

\section{Proof of Theorem \ref{th:kdelta}}

\ \indent Before presenting the proof, we first introduce two commonly used formulas in probability theory.

\begin{lemma}\cite[Chapter 9]{BH2019} \label{LTV}
Let $A$ and $B$ be random variables. Then
\begin{enumerate}
\item[(1).] (Law of Total Expectation)  $\E (A) = \E \Big(\E (A\mid B)\Big).$

\item[(2).] (Law of Total Variance) $\Var (A)=\Var \Big(\E (A\mid B)\Big) + \E\Big(\Var (A\mid B)\Big).$
\end{enumerate}
\end{lemma}

Now we turn to the proof of Theorem \ref{th:kdelta}.

\noindent\textit{Proof of Theorem \ref{th:kdelta}.
}
Let $G$ be an $n$-vertex graph with maximum degree $\Delta$, and let $I$ be an independent set of $G$ drawn from the hard-core model at fugacity $\lambda$.
Define
\[
\Phi(I)=\{v:\ v\notin I\text{ and } I\cup\{v\}\in\mathcal{I}(G)\},\qquad
\varphi(I)=|\Phi(I)|,\qquad X_G=|I|.
\]
When no confusion arises, we write $i_k$ for $i_k(G)$, $X$ for $X_G$ and $\varphi$ for $\varphi(I)$.
We set $i_{n+1} = 0$.
Since every independent set of size $k+1$ arises by adding one vertex to exactly $k+1$ independent sets of size $k$, we have
\begin{equation}\label{eq:sumphi}
\sum_{|I|=k}\varphi(I)=(k+1)\,i_{k+1}(G).
\end{equation}
Then we have
\begin{align}
\E (\varphi(I)) =& \sum_{I}\varphi(I)\Pr[I] =\sum_{I}\varphi(I)\frac{\lambda^{|I|}}{P_G(\lambda)} \notag\\
=& \frac{1}{P_G(\lambda)}\sum_{k=0}^{n}\lambda^k \sum_{|I|=k}\varphi(I)\notag\\
=&\frac{1}{P_G(\lambda)}\sum_{k=0}^{n} (k+1)i_{k+1}(G)\lambda^k \notag\\
=& \frac{\sum_{m=1}^{n} mi_{m}(G)\lambda^{m-1}}{P_G(\lambda)} = \frac{1}{\lambda} \mathbb E(X_G). \label{eq:Ephi}
\end{align}

Each vertex in $I$ has at most $\Delta$ neighbors that cannot be added to $I$.
So
\begin{align}\label{eq:varphiI}
\varphi(I)\geq n-\sum_{u\in I} d_G(u)-|I| \geq n-(\Delta+1) |I|.
\end{align}
Taking expectations and using \eqref{eq:Ephi} yields
\[
\frac{\E (X_G)}{\lambda}\geq n-(\Delta+1)\E (X_G),
\]
i.e.,
\begin{equation}\label{eq:EX_lb}
\E (X_G) \geq \frac{n\lambda}{1+(\Delta+1)\lambda}.
\end{equation}

By \eqref{eq:sumphi}, we know that
\begin{equation}\label{eq:sumphi2}
\sum_{|I|=k}k\varphi(I)=k(k+1)i_{k+1}(G).
\end{equation}
Using \eqref{eq:sumphi2}, we have
\begin{align}
\E  (X\varphi(I)) =& \sum_{I} X\varphi(I) \Pr[I] = \sum_{k=0}^{n}\sum_{|I|=k} k\varphi(I) \frac{\lambda^k}{P_G(\lambda)} \notag\\
=&\frac{1}{P_G(\lambda)}\sum_{k=0}^{n} \lambda^k\sum_{|I|=k}k\varphi(I) \notag\\
=& \sum_{k=0}^{n} \frac{k(k+1)i_{k+1}(G)\lambda^k}{P_G(\lambda)} \label{eqEXPhi}.
\end{align}
And
\begin{equation}
\begin{aligned}
\mathbb E(X^2-X)
&= \frac{\sum_{k=1}^{n}(k^2-k) i_k \lambda^k }{P_G(\lambda)}\notag \\
&\overset{m=k-1}{=} \frac{\sum_{m=0}^{n}m(m+1) i_{m+1} \lambda^{m+1} }{P_G(\lambda)} . \label{eqEX2-X}
\end{aligned}
\end{equation}
Combining the above formula with \eqref{eqEXPhi}, we have
\begin{equation}\label{eq:EXphi}
\E  (X\varphi(I)) = \frac{1}{\lambda}\mathbb E(X(X-1)).
\end{equation}
Therefore, by \eqref{eq:Ephi} and \eqref{eq:EXphi}, we have

\begin{align}
\Cov (X,\varphi(I)) =& \E (X\varphi) - \E (X)\E (\varphi) \notag\\
=& \frac{1}{\lambda} \Big( \E (X(X-1))-(\E (X))^2 \Big)\notag\\
=& \frac{1}{\lambda} \Big(\Var (X)-\E (X) \Big). \label{eq:CovXp}
\end{align}

Let $q_k:=\E (\varphi(I)\mid |I|=k)$,
for $k=1,2,\cdots,n$. Particularly, $q_0=n$ and $q_k=0$ if $i_k=0$.
For each index $k$ such that $i_k > 0$, we have
\begin{align}
q_k=&\sum_{I}\varphi(I) \Pr[I\mid |I|=k] = \sum_{|I|=k} \varphi(I) \frac{\frac{\lambda^k}{P_G(\lambda)}}{\frac{i_k(G) \lambda^k}{P_G(\lambda)}} \notag \\
=& \sum_{|I|=k} \frac{\varphi(I)}{i_k(G)} \label{eq:qkphi}\\
=& \frac{(k+1)i_{k+1}(G)}{i_k(G)}. \label{eq:qkik}
\end{align}

Adding a vertex $u \in \Phi(I)$ to an independent set $I$ can block at most $\Delta+1$ vertices from $\Phi(I)$,
hence $\varphi(I\cup\{u\})\geq \varphi(I)-(\Delta+1)$.
Summing over all $u\in\Phi(I)$ and then over all $I$ with $|I|=k$,
we have
\begin{equation}\label{eq:IUu}
\sum_{|I|=k} \sum_{u\in\Phi(I)} \varphi(I\cup\{u\}) \geq \sum_{|I|=k} \varphi(I)(\varphi(I)-(\Delta+1)).
\end{equation}
The left-hand side sums over all independent sets $I$ of size $k$ and vertices \(u \in \Phi(I)\), computing \(\varphi(I \cup \{u\})\).
For each independent set $J$ of size $k+1$, we have $J=I\cup\{u\}$ for some $I$ with $|I|=k$ and $u$.
Specifically, for a fixed $J$, there are exactly $k+1$ ways to choose a vertex $u\in J$  with $I=J\setminus \{u\}$ as an independent set of size $k$ and $u\in\Phi(I)$.
Each such pair $(I, u)$ contributes \(\varphi(J)\) to the left-hand side.
Thus \(\varphi(J)\) is counted $k+1$ times for each $J$.
So
\begin{equation}\label{eq:IJ}
\sum_{|I|=k} \sum_{u\in\Phi(I)} \varphi(I\cup\{u\}) = \sum_{|J|=k+1}(k+1)\varphi(J).
\end{equation}
Substituting \eqref{eq:IJ} into \eqref{eq:IUu}, we have
\begin{equation*}\label{eq:J>}
\sum_{|J|=k+1}(k+1)\varphi(J)\geq \sum_{|I|=k} \varphi(I)^2 - (\Delta+1)\sum_{|I|=k} \varphi(I).
\end{equation*}

Applying the \textit{Cauchy--Schwarz inequality}:
\[
\left(\sum_{|I|=k}\varphi(I)\cdot 1 \right)^2\leq \left( \sum_{|I|=k} \varphi(I)^2 \right) \cdot i_k(G),
\]
we obtain
\begin{align} \label{eq:phiJ}
(k+1)\sum_{|J|=k+1} \varphi(J)\geq \frac{\left(\sum_{|I|=k}\varphi(I) \right)^2}{i_k(G)}- (\Delta+1)\sum_{|I|=k}\varphi(I).
\end{align}
 By \eqref{eq:qkphi}, we have $\sum_{|I|=k}\varphi(I)=i_k\,q_k$ and $\sum_{|J|=k+1}\varphi(J)=i_{k+1}\,q_{k+1}$.
Substituting them into \eqref{eq:phiJ}, we have
\begin{equation}\label{ineq-new}
(k+1)i_{k+1}q_{k+1}\geq i_k q_k^{2}-(\Delta+1) i_k q_k.
\end{equation}
If \(i_k=0\), then \(q_k=0\).
If \(q_k=0\), then the desired inequality \eqref{eq:qk_rec} is trivial, since $q_{k+1}\ge 0 \ge q_k-(\Delta+1).$
Thus, we may assume that \(i_kq_k\neq 0\).
Using \eqref{eq:qkik}, dividing on both sides of Inequality (\ref{ineq-new}) by $i_k\,q_k$ yields
\begin{equation}\label{eq:qk_rec}
q_{k+1}\geq q_k-(\Delta+1).
\end{equation}

Write $q_X=\E (\varphi(I)\mid X)$.
By Lemma \ref{LTV} (1), we know that
\[
\E (\varphi(I)) = \E (\E (\varphi(I)\mid X)) = \E (q_X),
\]
\[
\E (X\varphi(I)) = \E (\E (X\varphi(I)\mid X)) =\E (X\E (\varphi(I)\mid X)) = \E (Xq_X).
\]
So we obtain that
\begin{equation}\label{eq:Covp=q}
\Cov (X,\varphi(I))=\Cov (X,q_X).
\end{equation}

Let $\eta_k:= q_k+ (\Delta+1)k$.
By \eqref{eq:qk_rec}, $\eta_k$ is non-decreasing in $k$.
Let $X'$ be an independent and identically distributed copy of $X$.
So $\E (X')=\E (X)$ and $\E (\eta_{X'})=\E (\eta_X)$.
Therefore
\[
0\leq \E \left[ (X-X')(\eta_X-\eta_{X'}) \right] =  2\left(\E (X\eta_X) -\E (X)\cdot\E (\eta_X) \right) = 2 \Cov (X,\eta_X),
\]
i.e.,
\[
\Cov (X,\eta_X)\geq 0.
\]
Expanding the covariance gives
\begin{align*}
\Cov (X,\eta_X)=&\Cov (X,q_X+ (\Delta+1)X) = \Cov (X,q_X) + (\Delta+1)\Cov (X,X) \\
=& \Cov (X,q_X) + (\Delta+1)\Var (X) \geq 0.
\end{align*}
Combining it with \eqref{eq:Covp=q}, we have
\begin{equation}\label{eq:cov_lb}
\Cov (X,\varphi(I))\geq -(\Delta+1)\Var (X).
\end{equation}
Substituting \eqref{eq:cov_lb} into \eqref{eq:CovXp}, we have
\[
\frac{1}{\lambda} (\Var (X)-\E (X)) \geq -(\Delta+1)\Var (X).
\]
By \eqref{eq:EX_lb}, we have
\begin{align}\label{eq:V>E}
\Var (X_G)\geq \frac{\E (X_G)}{1+(\Delta+1)\lambda} \geq \frac{n\lambda}{(1+(\Delta+1)\lambda)^2}.
\end{align}

Since $K_{\Delta+1}$ is the complete graph on $\Delta+1$ vertices,
its independence polynomial is
$P_{K_{\Delta+1}}(\lambda)=1+(\Delta+1)\lambda$.
Then
\[
V_{K_{\Delta+1}}(\lambda)
= \frac{1}{\Delta+1}\,\Var (X_{K_{\Delta+1}})
= \frac{\lambda}{\bigl(1+(\Delta+1)\lambda\bigr)^{2}}.
\]
By definition, we conclude
\[
V_G(\lambda)
= \frac{1}{n}\,\Var (X_G)
\geq \frac{\lambda}{\bigl(1+(\Delta+1)\lambda\bigr)^{2}}
= V_{K_{\Delta+1}}(\lambda). \qedhere
\]

Next we prove that equality holds if and only if $G$ is a disjoint union of copies of $K_{\Delta+1}$.
 Assume that $G$ satisfies $V_G(\lambda) = (\lambda)\big/\bigl(1+(\Delta+1)\lambda\bigr)^{2}$.
Then, by \eqref{eq:V>E}, we know that
\begin{align}\label{eq:EX=}
\E (X_G) = \frac{n\lambda}{1+(\Delta+1)\lambda}.
\end{align}
 Define
\begin{align*}\label{def:DI}
D(I):=\varphi(I)-\bigl(n-(\Delta+1)|I|\bigr).
\end{align*}
By \eqref{eq:varphiI}, we obtain $D(I)\geq0$.
 Taking expectation and using \eqref{eq:Ephi} and \eqref{eq:EX=}, we have
\[
\E (D(I)) =
\frac{\E (X_G)}{\lambda} -n+(\Delta+1)\E (X_G) =0.
\]
 Since $\lambda>0$, every independent set has positive probability under
the hard-core measure. Therefore,
$$ D(I)=0 $$
for every independent set $I$.
 For an arbitrary vertex $v$ in $G$, taking $I=\{v\}$, we have
\[
\varphi(I)=n-1-d_G(v).
\]
 Therefore,
\[
0=D(\{v\})=\Delta-d_G(v).
\]
Thus $G$ is $\Delta$-regular.

 Let $u$ and $v$ be two non-adjacent vertices  in $G$. Then $\{u,v\}$ is an independent set of $G$.
 So, we have
\begin{align*}
\varphi(\{u,v\}) &= n-2- | N_G(u)\cup N_G(v)| \notag \\
&=n-2- \left( 2\Delta - |N_G(u)\cap N_G(v)| \right).
\end{align*}
Consequently,
\begin{align}\label{Duv=0}
0=D(\{u,v\})=\lvert N_G(u)\cap N_G(v)\rvert.
\end{align}

 Thus no two non-adjacent vertices have a common neighbor.
We claim that every connected component of $G$ is complete.
Indeed, if a connected component is not complete, then there exist two non-adjacent vertices at distance~2, which have a common neighbor. This contradicts \eqref{Duv=0}.

 Hence $G$ is a $\Delta$-regular graph in which every connected component has exactly $\Delta+1$ vertices.
That is,  $G \cong sK_{\Delta+1}$ for some positive integer $s$.

 Conversely, when $G$ is a disjoint union of copies of $K_{\Delta+1}$, the conclusion clearly holds from the definition of $V_G(\lambda)$.
This completes the proof of the result.

\begin{remark}
As an immediate consequence of \eqref{eq:EX_lb}, we can obtain the result in Theorem~\ref{th:Ekdelta}. Indeed, since
$ nE_G(\lambda)=\mathbb{E}(X_G), $
we have
$$
E_G(\lambda) =\frac{\mathbb{E}(X_G)}{n}
\geq \frac{\lambda}{1+(\Delta+1)\lambda}
= E_{K_{\Delta+1}}(\lambda).
$$
\end{remark}

\section{Proofs of Corollaries \ref{th:kn}, \ref{th:dregular} and \ref{chi'}}

\ \indent In this section, based on the result of Theorem \ref{th:kdelta}, we focus on the proofs of Corollaries \ref{th:kn}, \ref{th:dregular} and \ref{chi'}.

\begin{proof}[Proof of Corollary \ref{th:kn}]
Let $G$ be an $n$-vertex graph and $\Delta$ be the maximum degree of $G$.
 By Theorem \ref{th:kdelta} and $\Delta\le n-1$, we obtain
\[
V_G(\lambda)\ge \frac{\lambda}{(1+(\Delta+1)\lambda)^2}
\ge \frac{\lambda}{(1+n\lambda)^2}
=V_{K_n}(\lambda).
\]
 Next we prove the condition under which equality holds.
If $V_G(\lambda)=V_{K_n}(\lambda)$, then we have
$\frac{\lambda}{(1+(\Delta+1)\lambda)^2} = \frac{\lambda}{(1+n\lambda)^2}$,
which means that $n=\Delta+1$.
By Theorem \ref{th:kdelta}, $G$ is the disjoint union of copies of $K_{\Delta+1}$ when $V_G(\lambda)=\frac{\lambda}{(1+(\Delta+1)\lambda)^2}$.
Therefore, since $n=\Delta+1$, $G$ has only one component, i.e., $G\cong K_n$.
This completes the proof.
\end{proof}

Similarly, we can prove Corollary \ref{th:dregular} as follows.

\begin{proof}[Proof of Corollary \ref{th:dregular}]
Let $G$ be a $d$-regular $n$-vertex graph.
Then the maximum degree of $G$ is $d$.
By Theorem \ref{th:kdelta}, we obtain
\[
V_G(\lambda)\ge \frac{\lambda}{(1+(d+1)\lambda)^2}
=V_{K_{d+1}}(\lambda).
\]
Equality holds if and only if $G$ is a disjoint union of copies of $K_{d+1}$.
This completes the proof of Corollary \ref{th:dregular}.
\end{proof}

Next we turn to the proof of Corollary \ref{chi'}.

\begin{proof}[Proof of Corollary \ref{chi'}]
Let $G$ be a graph with edge chromatic number $\chi'(G)$.
When $\chi'(G)=0$,  $G$ consists of isolated vertices, and the conclusion clearly holds.
Therefore, in the following proof, we assume that $\chi'(G)\ge 1$.
Since every proper edge coloring assigns distinct colors to the edges incident with any fixed vertex, we have
$$
\Delta(G)\leq\chi'(G).
$$
Let $k:=\chi'(G)$.
By Theorem \ref{th:kdelta}, we have
$$
V_G(\lambda) \geq \frac{\lambda}{\bigl(1+(\Delta+1)\lambda\bigr)^2} \geq \frac{\lambda}{\bigl(1+(k+1)\lambda\bigr)^2}.
$$

By Theorem 1 (ii) of \cite{BCC1967}, for $n\ge 2$, we obtain
\[
\chi'(K_n)=
\begin{cases}
n-1, & \text{if $n$ is even},\\
n,   & \text{if $n$ is odd}.
\end{cases}
\]

Let $G$ be a graph with  $\chi'(G)=k$ such that $V_G(\lambda)=\lambda\big/\bigl(1+(k+1)\lambda\bigr)^2$.
Then we know that $\Delta(G)=k$ and,  from Theorem \ref{th:kdelta}, $G$ must be a disjoint union of copies of $K_{k+1}$.
If $k$ is odd, then $k+1$ is even and
$$
\chi'(K_{k+1})=k.
$$
So the disjoint union of copies of $K_{k+1}$ attains the equality. If
$k\geq2$ is even, then
$$
\chi'(K_{k+1})=k+1,
$$
which contradicts the assumption that $\chi'(G)=k$. So no graph with edge chromatic number $k$ can attain the equality.
Therefore, equality holds if and only if $G$ is a disjoint union of copies of $K_{\chi'(G)+1}$ where $\chi'(G)$ is either zero or odd.
\end{proof}

\section{Proof of Theorem \ref{thm:c3freedegree}}

\ \indent In this section we will prove the result in Theorem \ref{thm:c3freedegree}.

Let $I$ be a random independent set of graph $G$ drawn from the hard-core model at fugacity $\lambda$.
We say a vertex $v$ is \textit{occupied} if $v\in I$ and \textit{unoccupied} otherwise.

 \begin{lemma}\label{le:c3freemaxdeg}\cite{DJPR2018}
Let $G$ be a triangle-free graph with maximum degree $\Delta$. Then, for any $\lambda > 0$,
\[
E_G(\lambda) \ge
\frac{\lambda}{1+\lambda}
\frac{W\!\left(\Delta\log(1+\lambda)\right)}{\Delta\log(1+\lambda)},
\]
where $W$ is the (positive real branch of the) Lambert $W$-function.
\end{lemma}

 Next we turn to the proof of  Theorem \ref{thm:c3freedegree}.

 \begin{proof}[Proof of Theorem \ref{thm:c3freedegree}]
Let $G$ be an $n$-vertex triangle-free graph with maximum degree $\Delta$.
When $\Delta(G)=0$, the conclusion clearly holds.
Therefore, in the following proof, we assume that $\Delta(G)\ge 1$.

A proper vertex coloring of the vertex set $V(G)$ is performed.
Since the maximum degree of $G$ is $\Delta$,
the vertex set can be partitioned into at most $\Delta+1$ chromatic classes:
\[
V(G)= S_1\cup S_2\cup\cdots\cup S_r, \qquad r\le \Delta+1.
\]
 Note that $S_i\in \mathcal{I}(G)$ for $i=1,2,\cdots,r$.
 Since
$$
\sum_{i=1}^{r}\sum_{v\in S_i}\Pr[v\in I] = \sum_{v\in V(G)}\Pr[v\in I],
$$
 by averaging, there exists a color class $S=S_i$ such that
$$
\sum_{v\in S}\Pr[v\in I] \geq \frac{1}{r}\sum_{v\in V(G)}\Pr[v\in I]
\geq \frac{1}{\Delta+1}\sum_{v\in V(G)}\Pr[v\in I].
$$
 By the definition of $E_G(\lambda)=\frac{1}{n}\sum_{v\in V(G)} \Pr[v\in I]$, we have
\begin{align}\label{specialS}
\sum_{v\in S}\Pr[v\in I] \ge \frac{nE_G(\lambda)}{\Delta+1}.
\end{align}

 Let
$$
J:=I\cap(V(G)\setminus S),\quad\text{and}
$$

$$
A(J)=\{u:\ u\in S \text{ and } N_G(u)\cap J=\varnothing\}.
$$

 Assume that $v\in S$.
 If $v\notin A(J)$, then $v$ has at least one neighbor belonging to $J$. So $v$ cannot be occupied by $I$.
 If $v\in A(J)$, then $v$ can be occupied.  Moreover, since $S$ contains no edges, any subset $B$ of $A(J)$ can be combined with $J$ to form an independent set. Consequently,
\[
\begin{aligned}
\Pr[I\cap S =B\mid J]
&= \frac{\lambda^{|J|+|B|}}
        {\displaystyle\sum_{C \subseteq A(J)} \lambda^{|J|+|C|}} \\[4pt]
&= \frac{\lambda^{|B|}}
        {\displaystyle\sum_{C \subseteq A(J)} \lambda^{|C|}} \\[4pt]
&= \frac{\lambda^{|B|}}{(1+\lambda)^{|A(J)|}}.
\end{aligned}
\]

Let $ p=\lambda/(1+\lambda). $
Then the above expression becomes
\begin{align}\label{eq:IB}
\Pr[I\cap S=B\mid J]=p^{|B|}(1-p)^{|A(J)|-|B|}.
\end{align}
Let $Y_v=\mathbf 1_{\{v \in I\}}$ for each $v\in A(J)$.
Then each $Y_v\mid J$ is a Bernoulli random variable with success probability $p$.
By the spatial Markov property of the hard-core model, we obtain
\[
\Pr [Y_v=1\mid J]=\Pr [v\in I \mid J] =p, \qquad \Pr [Y_v=0\mid J]=\Pr [v\notin I \mid J] =1-p.
\]

For every $B\subseteq A(J)$, the event $I\cap S=B$ specifies the joint configuration in which \(Y_v=1\) for \(v\in B\) and \(Y_v=0\) for \(v\in A(J)\setminus B\).
Hence, the left-hand side of \eqref{eq:IB} is a joint conditional probability, while the right-hand side of \eqref{eq:IB} is the product of the corresponding marginal conditional probabilities.
Since this factorization holds for every $B\subseteq A(J)$, conditionally on $J$, the random variables $(Y_v)_{v\in A(J)}$ are mutually independent Bernoulli random variables with parameter $p$.

 Let $X=|I|$.
The number of selected vertices outside $S$ is fixed as $|J|$, hence
\[
X = |J| + \sum_{v \in A(J)} \mathbf 1_{\{v \in I\}} =  |J| + \sum_{v \in A(J)} Y_v.
\]

 Since the variance of the fixed constant $|J|$ is zero,
we have
\begin{align*}
\Var (X\mid J) =& \Var(|J|+ \sum_{v\in A(J)} Y_v \mid J )\\
=&  \Var(\sum_{v\in A(J)} Y_v \mid J) \\
=& \sum_{v\in A(J)} \Var(Y_v \mid J) \\
=& |A(J)| p (1-p).
\end{align*}

 By Lemma \ref{LTV} (2), we have
\begin{align}
\Var(X) =& \E \bigl( \Var(X \mid J) \bigr) + \Var\bigl( \E[X \mid J] \bigr)\notag\\
\ge& p(1-p) \E (|A(J)|). \label{eq:varlow}
\end{align}

 For any $v \in S$, it is occupied only if it belongs to set $A(J)$.
Therefore, the probability that $v$ is occupied is:
\[
\Pr[v\in I] = p \Pr [v \in A(J)].
\]
 Summing over $v \in S$, we get
\begin{align}\label{eq:pvS}
\sum_{v \in S} \Pr[v\in I] =& p \sum_{v \in S} \Pr[v \in A(J)] \notag\\
=& p \sum_{v\in S}\E \left(\mathbf{1}_{v\in A(J)}\right) \notag\\
=& p \E \left( \sum_{v\in S} \mathbf{1}_{v\in A(J)} \right) \notag\\
=&  p \E (|A(J)|).
\end{align}
 Substituting \eqref{eq:pvS} into \eqref{eq:varlow}, we get
\begin{align*}
\operatorname{Var}(X)
&\geq p(1-p) \mathbb E (|A(J)|) \\
&= (1-p) \sum_{v \in S} \Pr[v\in I] \\
&= \frac{1}{1+\lambda} \sum_{v \in S} \Pr[v\in I].
\end{align*}

 By \eqref{specialS}, we obtain
\begin{align}\label{eq:Varfinal}
\operatorname{Var}(X) \ge \frac{1}{1+\lambda} \sum_{v \in S} \Pr[v\in I]
\ge \frac{n\,E_G(\lambda)}{(\Delta+1)(1+\lambda)}.
\end{align}

 Combining with Lemma \ref{le:c3freemaxdeg} and \eqref{eq:Varfinal}, we obtain
\begin{align*}
V_G(\lambda) =& \frac{\operatorname{Var}(X)}{n} \ge \frac{E_G(\lambda)}{(\Delta+1)(1+\lambda)}\\
\ge& \frac{\lambda}{(1+\lambda)^2(1+\Delta)}
\frac{W\!\left(\Delta\log(1+\lambda)\right)}{\Delta\log(1+\lambda)}.
\end{align*} This finishes the proof.
\end{proof}

\section*{Acknowledgments}
This work is partially supported by NNSF of China (No.\ 12271251).

\end{document}